\documentclass[12pt]{article}
\usepackage{amsmath,amssymb,amscd,amsthm,graphicx,enumerate,enumitem}

\title{Ancient Solutions of the Mean Curvature Flow}
\author{Robert Haslhofer and Or Hershkovits\thanks{O.H. has been partially supported by NSF grant DMS 1105656.}}
\date{}

\usepackage{amsfonts}

\numberwithin{equation}{section}
  \theoremstyle{plain}
 \newtheorem{theorem}[equation]{Theorem}

 \newtheorem{claim}[equation]{Claim}
 \theoremstyle{remark}
 \newtheorem{remark}[equation]{Remark}
 \theoremstyle{remark}
\theoremstyle{definition}
 \newtheorem{definition}[equation]{Definition}

\newcommand{\eps}{\varepsilon}
\newcommand{\al}{\alpha}

\providecommand{\abs}[1]{\lvert #1\rvert}

\newcommand{\R}{\mathbb{R}}

\begin{document}
\maketitle

\begin{abstract}
In this short article, we prove the existence of ancient solutions of the mean curvature flow that for $t\to 0$ collapse to a round point, but for $t\to -\infty$ become more and more oval: near the center they have asymptotic shrinkers modeled on round cylinders $S^j\times \mathbb{R}^{n-j}$ and near the tips they have asymptotic translators modeled on $\mathrm{Bowl}^{j+1}\times \mathbb{R}^{n-j-1}$. We also give a characterization of the round shrinking sphere among ancient $\alpha$-Andrews flows. Our proofs are based on the recent estimates of Haslhofer-Kleiner \cite{HK}. 
\end{abstract}

\section{Introduction}

In this article, we study ancient solutions of the mean curvature flow. Recall that a one-parameter family of embedded hypersurfaces ${M_t\subset \R^{n+1}}$ moves by mean curvature flow if the normal velocity at each point is given by the mean curvature vector. A solution is called ancient if it is defined on a time interval $(-\infty,T)$, $T\leq \infty$. Ancient solutions typically arise in the study of singularities and of high curvature regions (see e.g. \cite{Ham,Whi00,Whi03,Eck,HS3,HK,HK2}). They also arise in conformal field theory, where they describe the ultraviolet regime of the boundary renormalization group equation (see e.g. \cite{Leigh,Fateev,Bakas}).

Daskalopoulos, Hamilton and Sesum have obtained a complete classification of ancient solutions in the case of closed embedded curves \cite{DHS}.
In higher dimensions, based on formal matched asymptotics, Angenent recently conjectured the existence of ancient ovals \cite{Ang}, i.e. ancient solutions that for $t\to 0$ collapse to a round point, but for $t\to -\infty$ become more and more oval in the sense that they look like round cylinders $S^j\times \mathbb{R}^{n-j}$ near the central region and like translating solitons $\mathrm{Bowl}^{j+1}\times \mathbb{E}^{n-j-1}$ near the tips.
In fact, a variant of this conjecture has been proved already by White \cite[p. 134]{Whi03}.
Namely, by considering convex regions of increasing eccentricity and using a limiting argument, he proved the existence of ancient flows of compact, convex sets
that are not selfsimilar. Although not explicitly stated there, it of course follows from Huisken's monotonicity formula \cite{Hui2} and the arguments in \cite{Whi00,Whi03}, that the tangent flows at the singularity are shrinking spheres and the tangent flows at infinity are shrinking cylinders (since otherwise the whole flow would be selfsimilar).

The main purpose of this article is to carry out White's construction in more detail (including in particular the study of the geometry at the tips), and thus to give a rigorous and simple proof for the existence of ancient ovals, as conjectured by Angenent. We write $H$ for the mean curvature, $A$ for the second fundamental form, $\kappa_1\leq\ldots \leq\kappa_n$ for the principal curvatures, and $\mathrm{Bowl}$ for the unique rotationally invariant translating soliton \cite{AltWu}, normalized such that the mean curvature at the tip equals one.

\begin{theorem}[Existence of ancient ovals]\label{Main_Thm}
For every $1\leq j \leq n-1$, there exists an ancient solution $\{M_t\subset \mathbb{R}^{n+1}\}_{t\in (-\infty,0)}$ of the mean curvature flow with compact and strictly convex time slices that for $t\to 0$ converges to a round point and for $t\to -\infty$ has the following asymptotics:
\begin{itemize}
\item\label{MT_Asym_Cyl} asymptotic shrinker: for $\lambda\to\infty$ the parabolically rescaled flows $\lambda^{-1}M_{\lambda^{2}t}$ converge to the round shrinking cylinder $S^j(\sqrt{2j\abs{t}})\times \R^{n-j}$.
\item\label{MT_Trans} asymptotic translators: given any direction $v\in{0}_{\mathbb{R}^{j+1}}\times \mathbb{S}^{n-j-1}$, the blowdowns at the tip in direction $v$ (see Claim \ref{claim_translators}) converge to $\mathrm{Bowl}(v)^{j+1}_t\times \mathbb{E}^{n-j-1}$, where $\mathrm{Bowl}(v)^{j+1}_t\subset\R^{j+1}\times\langle v\rangle$ translates in direction $-v$ and $\mathbb{E}^{n-j-1}$ is the orthogonal complement of $\langle v\rangle$ in $\mathbb{R}^{n-j}$.        
\end{itemize}
Moreover, our solution is $\alpha$-Andrews noncollapsed for some $\alpha=\alpha(n)>0$ (see Def. \ref{And_Cond}), and it also has the following additional properties:
\begin{enumerate}[label=\emph{\alph*})]
\item\label{add_prop1} it is uniformly $(n-j+1)$-convex: $\liminf_{t\to -\infty}\inf_{M_t}\frac{\kappa_1+\ldots+\kappa_{n-j+1}}{H}>0$,
\item\label{add_prop2} it has unbounded rescaled diameter: $\lim_{t\to -\infty} |t|^{-1/2}\mathrm{diam}(M_t) = \infty$,
\item\label{add_prop3}the curvature decay is type II: $\limsup_{t\rightarrow -\infty} |t|^{1/2} \sup_{M_{t}} \abs{A}=\infty$.
\end{enumerate}
\end{theorem}

\begin{remark}
Note that \ref{add_prop3} follows from \ref{add_prop2} by integration.
\end{remark}

\begin{remark}
For $j=n-1$ ($n\geq 2$) our solutions are analogous to Perelman's example for three-dimensional Ricci flow \cite[Ex. 1.4]{Per2}.
\end{remark}

In physical terms, our solutions can be thought of as phase transitions between the sphere in the infrared and the cylinders in the ultraviolet. Our proof is based on the recent estimates of Haslhofer-Kleiner \cite{HK}. These estimates have been developed in the context of mean convex (i.e. $H\geq 0$) mean curvature flows satisfying the conclusion of Andrews' beautiful noncollapsing result \cite{And}. Let us now recall the definition:

\begin{definition}[Andrews condition {\cite[Def. 1.1]{HK}}]\label{And_Cond}
Let $\alpha>0$. A mean convex mean curvature flow $\{M_t\}$ is called $\alpha$-Andrews if for every $p\in M_t$ there are interior and exteriors balls tangent  at $p$ of radius at least $\frac{\alpha}{H(p)}$.
\end{definition}

\begin{remark}
There is a more general notion for weak solutions, but by \cite[Thm. 1.14]{HK} ancient $\alpha$-Andrews flows are automatically smooth until they become extinct. Also recall that, by the maximum principle, mean convexity and the Andrews condition are both preserved under mean curvature flow.
\end{remark}

For comparison, we also prove the following theorem which characterizes the round shrinking sphere in the class of ancient $\alpha$-Andrews flows.

\begin{theorem}[Characterization of the sphere]\label{Comp_Thm}
Let $\{M_t\subset \mathbb{R}^{n+1}\}_{t\in (-\infty,0)}$ be an ancient $\alpha$-Andrews flow and assume \emph{at least one} of the following conditions is satisfied:
\begin{enumerate}[label=\emph{\alph*}) ,ref=\ref{Comp_Thm}.\alph*]
\item\label{CT_No_Un_convex} it is uniformly convex: $\liminf_{t\to -\infty}\inf_{M_t}\frac{\kappa_1}{H}>0$,
\item\label{CT_Diam} it has bounded rescaled diameter: $\limsup_{t\to -\infty} |t|^{-1/2}\mathrm{diam}(M_t) < \infty$,
\item\label{CT_Curvature} the time slices are compact and the curvature decay is type I:\\
$\limsup_{t\to -\infty} |t|^{1/2}\sup_{ M_t}\abs{A}< \infty$.
\end{enumerate}
Then $\{M_t\}$ is a family of round shrinking spheres.
 \end{theorem}
 
 \begin{remark}
A related result for closed, ancient, convex solutions of the mean curvature flow has been announced recently by Huisken-Sinestrari \cite{HS_OWR}. Our proof based on the Andrews condition seems to be much shorter, though.
\end{remark}

Theorem \ref{Comp_Thm} shows that the additional properties \ref{add_prop1} -- \ref{add_prop3} in Theorem \ref{Main_Thm} are in a sense sharp. Namely, if any of them is strengthened slightly, then this forces the solution to be a family of round shrinking spheres.\\

\noindent\emph{Acknowledgements:} We thank Bruce Kleiner and Brian White for useful comments.

\section{Preliminaries}\label{preliminaries}

Three key ingredients in our proofs are the convexity estimate, the global convergence theorem and the structure theorem established in \cite{HK}. As usual, we use the notation $P(p,t,r)=B(p,r)\times (t-r^2,t]$ for the parabolic ball centered at $(p,t)\in\R^{n+1}\times \R$, of size $r>0$.

\begin{theorem}[Convexity estimate {\cite[Thm. 1.10]{HK}}]\label{thm_conv}
For all $\eps>0$, $\al >0$, there exists $\eta=\eta(\eps,\alpha)<\infty$ with the following property.
If $M_t$ is an $\al $-Andrews flow in a parabolic ball  $P(p,t,\eta\, r)$ centered at a point 
$p\in M_t$ with $H(p,t)\leq r^{-1}$, 
then
\begin{equation}
\kappa_1(p,t)\geq -\eps r^{-1}.
\end{equation}
\end{theorem}
 
\begin{theorem}[Global convergence {\cite[Thm. 1.12]{HK}}]\label{Glob_Conv}
Let $M_t^k$ be a sequence of $\al $-Andrews flows with $\sup_k H(0,0)<\infty$ that is defined in parabolic balls $P(0,0,r_k)$ centered at $0\in M_0^k$ with $r_k\to\infty$. Then there exists a smoothly convergent subsequence, $M_t^{{k_\ell}}\to M_t^\infty$ in $C^\infty_\textrm{loc}$ on $\R^{n+1}\times (-\infty,0]$.
\end{theorem}

\begin{theorem}[Structure theorem {\cite[Part of Thm. 1.14]{HK}}]\label{structure_thm}
Ancient $\alpha$-Andrews flows in $\R^{n+1}$ are smooth until they become extinct. The only selfsimilarly shrinking ones are the sphere, the cylinders and the plane.
\end{theorem}

Let us also recall the rigidity case of Hamilton's Harnack inequality. Since we don't know a-priori that the limits obtained using Theorem \ref{Glob_Conv} have bounded curvature, this requires some minor adjustments.

\begin{theorem}[Rigidity of Hamilton's Harnack inquality {\cite[Thm. B]{Ham}}]\label{harn_equ}
Let $\{M_t\subset \R^{n+1}\}_{t\in(-\infty,\infty)}$ be a convex eternal mean curvature flow that satisfies Hamilton's Harnack inequality. Assume that the mean curvature attains a critical value at a point in space-time. Then $M_t$ is a translating soliton.
\end{theorem}

\begin{remark}
Hamilton assumes that the mean curvature attains a maximum. However, his discussion of the equality case of the maximum principle goes through verbatim for critical values. Also, note that the strict maximum principle is local, and thus does not require curvature bounds.
\end{remark}

\section{Proof of Theorem \ref{Main_Thm}}

Fix $1\leq j\leq n-1$. Although this is not strictly necessary, it is convenient to construct the solution in a $\mathrm{O}_{j+1}\times \mathrm{O}_{n-j}$ symmetric way.\\

For every $\ell\in \mathbb{N}$ we construct a hypersurface $M^\ell$ as follows. We take the cylinder $S^j\times [-\ell,\ell]^{n-j}$ of length $2\ell$ and cap it off in a rotationally symmetric, strictly convex way, say at a scale of length one. We choose the cap off function independently of $\ell$, and thus observe:
\begin{itemize}
\item The hypersurfaces $M^\ell$ are uniformly $n-j+1$ convex, i.e. $\kappa_1+\ldots+\kappa_{n-j+1}\geq \beta H$ for some $\beta=\beta(n)>0$ uniformly for all $\ell$.  
\item There exists an $\alpha=\alpha(n)>0$ such that $C^\ell$ is $\alpha$-Andrews for all $\ell$.
\end{itemize}

Let $M^\ell_t$ be the mean curvature flow starting at $M^\ell$ at $t=0$. By Huisken's classical theorem \cite{Hui1} (see also Andrews \cite[Rmk. 6]{And}) the flow collapses to a round point in finite time.
Using spheres as interior barriers and cylinders as exterior barriers, we see that this time is comparable to one.\\

Now, let $\{\hat{M}^\ell_t\}_{t\in[T_\ell,0)}$ (where $T_\ell< -1$ denotes the new initial time) be the sequence of solutions obtained by parabolically rescaling $M^\ell_t$ and shifting the time parameter such that:
\begin{itemize}
\item The flow becomes extinct at $t=0$.
\item The ratio of the major radius $a(t)=max_{x\in \hat{M}^\ell_t}(\sum_{i=j+2}^{n+1}x_{i}^{2})^{1/2}$  and the minor radius $b(t)=max_{x\in \hat{M}^\ell_t}(\sum_{i=1}^{j+1}x_{i}^{2})^{1/2}$ equals $2$ for the \emph{first} time.
\end{itemize}

\begin{claim}\label{claim_diam}
There exists $C<\infty$ independent of $\ell$ such that 
\begin{equation}
C^{-1}\leq \textrm{diam}(\hat{M}^\ell_{-1})\leq C.
\end{equation}
\end{claim}

\begin{proof}
Since the flow becomes extinct in one unit of time, the lower bound on the diameter follows from comparison with spheres. Using convexity and $\tfrac{a(-1)}{b(-1)}=2$, this also implies an upper bound on the diameter.
\end{proof}

\begin{claim}
$lim_{\ell\rightarrow\infty}T_\ell=-\infty$
\end{claim}
\begin{proof}
Consider a time $t_{0}<-1$. Using convexity and $a(t_{0})/b(t_{0}) \geq 2$, we can put a sphere of radius $b(t_{0})/{4}$ inside $\hat{M}_{t_0}$ at distance  $a(t_{0})/2$ from the origin. Thus, it takes $a(t)$ a time period of at least $b(t_{0})^2/32n$ to decrease by a factor one-half. On the other hand, $b(t)$ decreases with time and by Claim \ref{claim_diam}, we know that $b(t_{0})\geq \delta$ for some $\delta>0$. Thus, it takes the quotient $\tfrac{a(t)}{b(t)}$ a time period of at least $\delta^2/32n$ to decrease by a factor one-half.  Since $a(T_\ell)/b(T_\ell)\to\infty$ and $a(-1)/b(-1)=2$, the claim follows. 
\end{proof}

Now, by the global convergence theorem (Thm. \ref{Glob_Conv}), we see that the sequence $\hat{M}^\ell_t$ subconverges smoothly to an ancient $\alpha$-Andrews flow $M_{t}$ with compact and strictly convex time slices that becomes extinct in a round point at the origin at time $t=0$.\footnote{To see this, one applies the global convergence theorem at points $(p_\ell^k,t_\ell^k)$ with $t_\ell^k\in [-1/k,0]$ , $d(p_\ell^k,0) \leq C$  and $H(p_\ell^k,t_\ell^k)\leq Ck$ and passes to a diagonal subsequence, c.f. \cite[Lem. 3.9]{HK}.} By the condition $\tfrac{a(-1)}{b(-1)}=2$ it is certainly not the sphere. As we obtained $M_{t}$ as a limit of rescalings of $O_{j+1}\times O_{n-j}$ symmetric flows which are uniformly $(n-j+1)$-convex, $M_t$ has that symmetry and convexity as well. To finish the proof of Theorem \ref{Main_Thm}, it remains to establish the other claimed properties for $t\to -\infty$.\\

Let us start with the asymptotic shrinker. Since the flow becomes extinct in $0$ at time $0$, by comparison with spheres we have $d(0,M_t)\leq\sqrt{2n\abs{t}}$ and thus, as before (cf. \cite[Lem. 3.9]{HK}), points close to $(0,0)$ with controlled curvature. By the global convergence theorem (Thm. \ref{Glob_Conv}) and Huisken's monotonicity formula \cite{Hui2} (see also \cite[App. B]{HK}) the flows $\lambda^{-1}M_{\lambda^{2}t}$ subconverge to a nonempty selfsimilar shrinking flow $N_t$. By the structure theorem (Thm. \ref{structure_thm}), $N_t$ must be a family of shrinking spheres, cylinders or a plane. Due to symmetry reasons and the condition $a(t)/b(t)\geq 2$ for all $t\leq -1$, it must be $S^j(\sqrt{2j\abs{t}})\times \R^{n-j}$. In particular, the limit is unique and thus a full limit. The properties \ref{add_prop2} and \ref{add_prop3} in Theorem \ref{Main_Thm} follow also.\\

Finally, let us discuss the asymptotic translators. Fix a direction $v\in{0}_{\mathbb{R}^{j+1}} \times \mathbb{S}^{n-j-1}$.
Let $p_t$ be the unique point at the tip of $M_t$ in direction $v$, i.e. the unique point in $M_t$ that can be written as $\mu v$ for some $\mu>0$.
We now perform the modified type II blow-down as follows. Pick times $t_k$ such that
\begin{equation}\label{Type_II_Point_Picking}
\abs{t_k}^{1/2}H(p_{t_k},t_{k})=\max_{t\in[-k,-1]}\abs{t}^{1/2}H(p_{t}).
\end{equation}

\begin{claim}\label{claim_translators}
Let $\hat{M}_{\hat{t}}^k$ be the sequence of flows obtained by shifting $(p_{t_k},t_{k})$ to the origin and normalizing $\lambda_k=\mathrm{H}(p_{t_k},t_{k})$ to one,
explicitly $\hat{M}_{\hat{t}}^k=\lambda_k \cdot (M_{\hat{t}/\lambda_k^2+t_k}-p_{t_k})$.
Then $\hat{M}_{\hat{t}}^k$ converges to $\mathrm{Bowl}(v)^{j+1}_{t}\times\mathbb{E}^{n-j-1}$.
\end{claim}

\begin{proof}
Let $r(t)=d(p_{t},0)$. By the proof of property \ref{add_prop2} and \ref{add_prop3} we actually have 
$\lim_{t\rightarrow -\infty}\abs{t}^{-1/2}r(t)=\infty$,
and
$\limsup_{t\rightarrow-\infty}\abs{t}^{1/2}\mathrm{H}(p_{t})=\infty$.
In particular, $t_{k}\rightarrow-\infty$ as $k\rightarrow\infty$. By construction the flows $\hat{M}_{\hat{t}}^k$ satisfy $H(0,0)=1$ and are defined for $-\infty<\hat{t}<\lambda_k^2\abs{t_k}\to\infty$. Moreover, by condition (\ref{Type_II_Point_Picking}) we have:  
\begin{equation}\label{Lower_Curv_After}
\hat{H}^{2}(\hat{p},\hat{t})=
\frac{H^{2}\left(p,\hat{t}/\lambda_k^2+ t_{k}\right)}{\lambda_k^2}\leq
\left|\frac{t_{k}}{\hat{t}/\lambda_k^2+ t_{k}}\right| =\left|\frac{1}{1-\hat{t}/(\lambda_k^2\abs{t_k})}\right| \rightarrow1
\end{equation} 
for $p$ at the tip and all $\hat{t}$ with $0\leq \hat{t}<\lambda_k^2(\abs{t_{k}}-1)$.

By the global convergence theorem (Thm. \ref{Glob_Conv}),  $\hat{M}_{\hat{t}}^k$ subconverges to an eternal convex $\alpha$-Andrews flow $\{N_t\}_{t\in(-\infty,\infty)}$ with $H>0$ everywhere (as it is not flat). Note that
\begin{equation}
\hat{r}_{k}=\lambda_k r(t_{k}) =\lambda_k\abs{t_k}^{1/2}\abs{t_k}^{-1/2}r(t_{k}) \rightarrow\infty.
\end{equation}
Thus, $\kappa_1=\ldots=\kappa_{n-j-1}=0$, and by the strict maximum principle, the limit splits orthogonally as $N_t=B_t \times \mathbb{E}^{n-j-1}$, where $B_t$ is strictly convex.

Note that $B_t$ is $O_{j+1}$ symmetric. In particular, $\nabla H=0$ at the origin. Since $B_t$ arises as a smooth limit of compact solutions it satisfies Hamilton's Harnack inequality \cite[Thm. A]{Ham}, in particular $\partial_t H\geq 0$. Together with equation (\ref{Lower_Curv_After}) this implies $\partial_{t}H=0$ at the origin (for all times $t>0$). Thus, by the equality case of Hamilton's Harnack inequality (Thm. \ref{harn_equ}), $B_t$ must be a translating soliton. By rotational symmetry, it must be the bowl. Finally, due to uniqueness, the subsequential limit is actually a full limit.
\end{proof}

\begin{remark}
We emphasize that contrary to the standard type II blow-up procedure (see e.g. \cite[Sec. 4]{HS2}), we did not pick points $(p,t)$ with maximal $\abs{t}^{1/2}\mathrm{H}(p,t)$ over all times and points in those time slots. Since we didn't know a-priori the curvature is maximal at the tip, we needed the full strength of the global convergence theorem (Thm. \ref{Glob_Conv}) to pass to a smooth limit.
\end{remark}

\section{Proof of Theorem {\ref{Comp_Thm}}}
By the convexity estimate (Thm. \ref{thm_conv}) ancient $\alpha$-Andrews flows are convex.
Now, arguing as in the proof of Theorem \ref{Main_Thm}, we can find an asymptotic shrinker that must either be a round shrinking cylinder $S^j\times \R^{n-j}$ or a round shrinking sphere $S^n$. However, any of the assumptions a) -- c) excludes the cylinders (note that c) implies b) by integration). Thus, by the equality case of Huisken's monotonicity formula \cite{Hui2} (see also \cite[App. B]{HK}) the flow $\{M_t\}$ must be a family of round shrinking spheres.

\bibliographystyle{alpha}
\bibliography{ancient}

\vspace{10mm}
{\sc Robert Haslhofer and Or Hershkovits, Courant Institute of Mathematical Sciences, New York University, 251 Mercer Street, New York, NY 10012, USA}\\

\emph{E-mail:} robert.haslhofer@cims.nyu.edu, or.hershkovits@cims.nyu.edu

\end{document}